\documentclass[amssymb, 11pt]{amsart}
\usepackage{amsmath}
\usepackage{amsthm}
\usepackage{amsfonts}

\usepackage{latexsym}

\newlength{\standardunitlength}
\newtheorem{prop}{Proposition}[section]
\newtheorem{definition}[prop]{Definition}

\newtheorem{cor}[prop]{Corollary}
\newtheorem{theorem}[prop]{Theorem}

\newtheorem{remark}[prop]{Remark}

\begin{document}

\title [The distribution of fixed vectors] {On the distribution of the number of fixed vectors
for the finite classical groups}

\author{Jason Fulman}
\address{Department of Mathematics\\
University of Southern California\\
Los Angeles, CA 90089-2532} \email{fulman@usc.edu}

\author{Dennis Stanton}
\address{Department of Mathematics\\
University of Minnesota\\
Minneapolis, MN 55455} \email{stanton@math.umn.edu}


\date{May 23, 2015}


\begin{abstract} Motivated by analogous results for the symmetric group and
compact Lie groups, we study the distribution of the number of fixed vectors of
a random element of a finite classical group. We determine the limiting
moments of these distributions, and find exactly how large the rank of the
group has to be in order for the moment to stabilize to its limiting value.
The proofs require a subtle use of some q-series identities. We also point
out connections with orthogonal polynomials.
\end{abstract}

\maketitle

\section{Introduction}

In an influential and widely cited paper \cite{DS}, Diaconis and Shahshahani study the fixed
points of random permutations, and the trace of elements of compact Lie groups.
The following two theorems are special cases of their results:

\begin{theorem} \label{poisson} Let $\pi$ be uniformly distributed on the symmetric
group $S_n$, and let $Z(\pi)$ denote the number of fixed points of $\pi$. Then for any
natural number $j$, and $n \geq j$, the expected value of $Z^j$ is equal to
the Bell number $B_j$, the number of partitions of a set of size $j$. Note that
$B_j$ is the $j$th moment of a Poisson distribution with mean $1$, so that
$Z$ approaches a Poisson distribution with mean $1$ as $n \rightarrow \infty$.
\end{theorem}

\begin{theorem} \label{normal} Let $g$ be chosen from the Haar measure of the
orthogonal group $O(n,R)$, and let $Z(g)$ denote the trace of $G$. Then for any
natural number $j$, and $n \geq j$, the expected value of $Z^j$ is equal to
$0$ if $j$ is odd, and to $(j-1)(j-3) \cdots 1$ if $j$ is even. Note that
these moments are the same as the $j$th moment of a standard normal random
variable, so that $Z$ approaches a standard normal random variable as
$n \rightarrow \infty$.
\end{theorem}

It is natural to seek analogs of these results for the finite classical groups,
and the main purpose of this paper is to provide such analogs. For example
if $Z$ denotes the number of fixed vectors of a random element of $GL(n,q)$,
then for all natural numbers $j$ and $n \geq j$, the expected value of $Z^j$
turns out to be equal to the $j$th Galois number, that is the number of subspaces of a
$j$-dimensional vector space over the finite field $F_q$. Moreover as $n
\rightarrow \infty$, the chance that $Z=q^k$ approaches
\begin{equation} \label{fir} \prod_{r \geq 1} \left( 1 - \frac{1}{q^r} \right)
\frac{1}{q^{k^2} (1-1/q)^2 (1-1/q^2)^2 \cdots (1-1/q^k)^2} .\end{equation} This
limiting distribution can be thought of as a $q$-analogue of the Poisson distribution.

The starting point of our work is a beautiful, 80 page paper of Rudvalis and Shinoda
\cite{RS}, which was written in 1988 and unfortunately was never published.
They use Mobius inversion and a very heavy dose of combinatorics to determine,
for each finite classical group $G$, and for each integer $k$, the probability that the fixed space of
a random element of $G$ is $k$-dimensional. For example, they show that the chance that a random
element of $GL(n,q)$ has a $k$-dimensional fixed space is equal to:
\begin{equation} \label{fort} \frac{1}{|GL(k,q)|} \sum_{i=0}^{n-k} \frac{(-1)^i q^{{i \choose 2}}}{q^{ki} |GL(i,q)|}.
\end{equation} They also find formulae for limiting distributions such as \eqref{fir}.

The main purpose of the current paper is to use explicit formulas such as \eqref{fort} to study
the moments of the distribution of fixed vectors of random elements of finite classical
groups. Looking at \eqref{fort}, it is not obvious that it defines a probability distribution,
or how to compute its moments. We give a unified approach to such matters, for all finite
classical groups (general linear, unitary, symplectic, orthogonal) in both odd and even
characteristic. Our proofs are not difficult to follow, but did take us a while to discover.
Other proofs of some of our results can be found in unpublished thesis work \cite{F0} of the
first named author, but the approach here is more unified and gives sharp results in all
cases.

To close the introduction, we mention three reasons why our results may be of interest.
First, some researchers in number theory study ``Cohen-Lenstra heuristics'', and need
information about the distribution of fixed vectors of random elements of finite classical
groups; see \cite{W} for $GL(n,q)$ and \cite{Ma} for the case of finite symplectic groups.
Second, in the case of the symmetric groups, stability of moments of fixed points
(and more generally $i$-cycles), has applications to representation stability in
cohomology and asymptotics for families of varieties over finite fields; see Section
3.4 of \cite{CEF} for details. We are optimistic that our results for finite classical
groups might have similar applications. Third, the moment calculations of Diaconis
and Shahshahani \cite{DS} for compact Lie groups are celebrated in the random matrix
community; there are various other approaches to their work \cite{St}, \cite{PV},
as well as applications of their moment calculations to studying linear functionals of eigenvalues of random
matrices \cite{DE}, \cite{J}. Rains \cite{Ra} connects moments of traces in compact
Lie groups with longest increasing subsequence problems. It is reasonable to hope that the
study of moments for finite classical groups may also be fruitful.

The organization of this paper is as follows. Section \ref{RSreview} recalls formulae of
Rudvalis and Shinoda \cite{RS}. These explicit formulae are crucial to approach, and since
their 1988 preprint never appeared, we are forced to record some of their results. In any
case, we do get regular requests for a copy of the preprint \cite{RS}, so recording
these formulae should be helpful to other mathematicians. Section \ref{identities} collects
some $q$-series identities which we will use. These allow us to treat all the finite
classical groups in a unified way. Section \ref{main} contains our main results:
the exact determination of how large the rank of the group has to be in order for the
moments to stabilize to their limiting values (which we also determine). Section
\ref{orthog} uses orthogonal polynomials to give another calculation of the limiting
values of the moments.

\section{Results of Rudvalis and Shinoda} \label{RSreview}

The purpose of this section is to recall some results of Rudvalis and Shinoda,
dating back to 1988.

We begin with $GL(n,q)$. Recall that $|GL(n,q)|=q^{{n \choose 2}} \prod_{i=1}^n (q^i-1)$.

\begin{theorem}
\label{RSGL}

\begin{enumerate}

\item The chance that an element of $GL(n,q)$ has a $k$-dimensional fixed space is
equal to \[ \frac{1}{|GL(k,q)|} \sum_{i=0}^{n-k} \frac{(-1)^i q^{{i \choose 2}}} {q^{ki} |GL(i,q)|} .\]

\item For $k$ fixed, the $n \rightarrow \infty$ limiting proportion of elements of $GL(n,q)$ with a $k$-dimensional
fixed space is equal to \[ \prod_{r \geq 1} \left( 1 - \frac{1}{q^r} \right) \frac{1}{q^{k^2} \left( 1-1/q \right)^2 \left( 1-1/q^2 \right)^2 \cdots \left( 1-1/q^k \right)^2} .\]

\end{enumerate}

\end{theorem}

{\it Remarks:}
\begin{enumerate}
\item Part 2 of Theorem \ref{RSGL} follows from part 1 of Theorem \ref{RSGL} and an identity of Euler.

\item The proof of part 1 of Theorem \ref{RSGL} used Mobius inversion and delicate combinatorics. For a different proof,
the reader can consult \cite{F1}.

\end{enumerate}

Next we treat $U(n,q)$. Recall that $|U(n,q)|= q^{{n \choose 2}} \prod_{i=1}^n (q^i-(-1)^i),$ and that we view $U(n,q)$ as a subgroup of $GL(n,q^2)$.

\begin{theorem} \label{RSU}

\begin{enumerate}

\item The chance that an element of $U(n,q)$ has a $k$-dimensional fixed space is
equal to

\[ \frac{1}{|U(k,q)|} \sum_{i=0}^{n-k} \frac{(-1)^i (-q)^{{i \choose 2}}}{(-q)^{ki} |U(i,q)|}.\]

\item For $k$ fixed, the $n \rightarrow \infty$ limiting proportion of elements of $U(n,q)$ with a $k$-dimensional
fixed space is equal to \[ \prod_{r \geq 0} \left( 1+\frac{1}{q^{2r+1}} \right)^{-1} \frac{1}{q^{k^2} \left( 1-1/q^2 \right) \left( 1-1/q^4 \right) \cdots \left( 1-1/q^{2k} \right)} .\]

\end{enumerate}

\end{theorem}

{\it Remarks:}
\begin{enumerate}
\item Part 2 of Theorem \ref{RSU} follows from part 1 of Theorem \ref{RSU} and an identity of Euler.

\item The proof of part 1 of Theorem \ref{RSU} used Mobius inversion and delicate combinatorics. For a different proof,
the reader can consult \cite{F1}.

\end{enumerate}

Next we treat symplectic groups. Recall that \[ |Sp(2n,q)| = q^{n^2} \prod_{i=1}^n (q^{2i}-1).\]

\begin{theorem} \label{RSSp}
\begin{enumerate}
\item The proportion of elements of $Sp(2n,q)$ with a $2k$-dimensional fixed space is equal to
\[ \frac{1}{|Sp(2k,q)|} \sum_{i=0}^{n-k} \frac{(-1)^i q^{i(i+1)}}{|Sp(2i,q)| q^{2ik}} .\]

\item The proportion of elements of $Sp(2n,q)$ with a $2k+1$-dimensional fixed space is equal to
\[ \frac{1}{|Sp(2k,q)|q^{2k+1}} \sum_{i=0}^{n-k-1} \frac{(-1)^i q^{i(i+1)}}{|Sp(2i,q)| q^{2i(k+1)}} .\]

\item For $k$ fixed, the $n \rightarrow \infty$ limiting proportion of elements of $Sp(2n,q)$ with a $k$-dimensional
fixed space is equal to
\[ \prod_{r \geq 1} \left(1 + \frac{1}{q^r} \right)^{-1} \frac{1}{q^{(k^2+k)/2} (1-1/q) (1-1/q^2) \cdots (1-1/q^k)} .\]

\end{enumerate}
\end{theorem}

{\it Remarks:}
\begin{enumerate}
\item Part 3 of Theorem \ref{RSSp} follows from parts 1 and 2 of Theorem \ref{RSSp} and an identity of Euler.

\item The proof of parts 1 and 2 of Theorem \ref{RSSp} used Mobius inversion and delicate combinatorics. For a different proof,
the reader can consult \cite{F2} for the case of odd characteristic, and \cite{FG} for the case of even characteristic.
\end{enumerate}

Next we treat the orthogonal groups. Note that it is not necessary to treat odd dimensional orthogonal groups $O(2n+1,q)$ in
even characteristic. Indeed, such groups are isomorphic to the symplectic groups $Sp(2n,q)$, and an element in
$Sp(2n,q)$ has a $k$-dimensional fixed space if and only if the corresponding element in $O(2n+1,q)$ has a $k+1$-dimensional
fixed space.

\begin{theorem} \label{RSOodd} Suppose that $q$ is odd.
\begin{enumerate}
\item The proportion of elements of $O(2n+1,q)$ with a $2k$-dimensional fixed space is equal to
\begin{eqnarray*}
& & \frac{1}{2} \frac{1}{q^{2k^2-k}(1-1/q^2)(1-1/q^4) \cdots (1-1/q^{2k})} \\
 & & \cdot \sum_{i=0}^{n-k} \frac{(-1)^i}{q^{i^2+2ik} (1-1/q^2)(1-1/q^4) \cdots (1-1/q^{2i})}.
\end{eqnarray*}

\item The proportion of elements of $O(2n+1,q)$ with a $2k+1$-dimensional fixed space is equal to
\begin{eqnarray*}
& & \frac{1}{2} \frac{1}{q^{2k^2+k}(1-1/q^2)(1-1/q^4) \cdots (1-1/q^{2k})} \\
& & \cdot \sum_{i=0}^{n-k} \frac{(-1)^i}{q^{i^2+2i(k+1)} (1-1/q^2)(1-1/q^4) \cdots (1-1/q^{2i})}.
\end{eqnarray*}

\item For $k$ fixed, the $n \rightarrow \infty$ limiting proportion of elements of $O(2n+1,q)$ with a $k$-dimensional
fixed space is equal to
\[ \prod_{r \geq 0} \left(1 + \frac{1}{q^r} \right)^{-1} \frac{1}{q^{(k^2-k)/2} (1-1/q) (1-1/q^2) \cdots (1-1/q^k)} .\]

\end{enumerate}
\end{theorem}

{\it Remarks:}
\begin{enumerate}
\item Part 3 of Theorem \ref{RSOodd} follows from parts 1 and 2 of Theorem \ref{RSOodd} and an identity of Euler.

\item The proof of parts 1 and 2 of Theorem \ref{RSOodd} used Mobius inversion and delicate combinatorics. For a different proof,
the reader can consult \cite{F2}.
\end{enumerate}

Finally, we consider even dimensional orthogonal groups. Note that the formulas are the same in odd and even characteristic.

\begin{theorem} \label{RSOeven}

\begin{enumerate}
\item The proportion of elements of $O^{\pm}(2n,q)$ with a $2k$-dimensional fixed space is equal to

\begin{eqnarray*} & & \frac{q^k}{2|GL(k,q^2)|} \sum_{i=0}^{n-k} \frac{(-1)^i}{q^{(2k-1)i} (q^{2i}-1) \cdots (q^4-1)(q^2-1)}\\
 & & \pm \frac{1}{2} \frac{(-1)^{n-k}}{q^{2k(n-k)} |GL(k,q^2)| (q^{2(n-k)}-1) \cdots (q^4-1)(q^2-1)}. \end{eqnarray*}

\item The proportion of elements of $O^{\pm}(2n,q)$ with a $2k+1$-dimensional fixed space is
\[ \frac{1}{2 q^k |GL(k,q^2)|} \sum_{i=0}^{n-k-1} \frac{(-1)^i}{q^{i^2+2(k+1)i} (1-1/q^2)(1-1/q^4) \cdots (1-1/q^{2i})} .\]

\item For $k$ fixed, the $n \rightarrow \infty$ limiting proportion of elements of $O^{\pm}(2n,q)$ with a $k$-dimensional
fixed space is equal to
\[ \prod_{r \geq 0} \left(1 + \frac{1}{q^r} \right)^{-1} \frac{1}{q^{(k^2-k)/2} (1-1/q) (1-1/q^2) \cdots (1-1/q^k)} .\]
\end{enumerate}

\end{theorem}

{\it Remarks:}
\begin{enumerate}
\item Note that for Theorem \ref{RSOeven}, the third part follows from parts 1 and 2, and an identity of Euler.

\item The proof of parts 1 and 2 of Theorems \ref{RSOeven} used Mobius inversion and delicate combinatorics.
For a different proof,
the reader can consult \cite{F2} for the case of odd characteristic, and \cite{FST} for the case of even characteristic.
\end{enumerate}

\section{$q$-series identities}
\label{identities}

The purpose of this section is to collect the
$q$-series identities which will be used in the proofs of our
main theorems.

We use the standard notation \cite[(I.1),(I.42)]{GaRa} for the $q$-shifted
factorial and the $q$-binomial coefficient
$$
 \begin{aligned}
 (A;q)_n &=(1-A)(1-Aq)\cdots(1-Aq^{n-1}), \\
  \binom{n}{k}_q &= \frac{(q^n-1) \cdots (q^{n-k+1}-1)}{(q^k-1) \cdots (q-1)}
=\frac{(q^n;q^{-1})_k}{(q;q)_k}.
\end{aligned}
$$

Note that all of the numerical quantities in Section \ref{RSreview} can be written in terms of $q$-shifted factorials.
\begin{prop} We have
$$
\begin{aligned}
|GL(n,q)|=&(q^n-1)\cdots (q^n-q^{n-1})=(-1)^nq^{\binom{n}{2}}(q;q)_n,\\
|U(n,q)|=&q^{\binom{n}{2}}\prod_{i=1}^n (q^i-(-1)^i)=Q^{\binom{n}{2}}(Q;Q)_n,\quad Q=-q,\\
|Sp(2n,q)|=&q^{n^2}(q^2-1)\cdots (q^{2n}-1)=(-1)^nR^{n^2/2}(R;R)_n, \quad R=q^2.
\end{aligned}
$$
\end{prop}

A certain double sum occurs in each of our results; we next define a general such sum.
\begin{definition}
\label{doubledefn}
Let
$$
D_n(P,X,Y)=\sum_{k=0}^n \frac{(-X)^k}{ P^{\binom{k}{2}}(P;P)_k}
\sum_{i=0}^{n-k} \frac{P^{-ik}Y^i}
{(P;P)_{i}}.
$$
\end{definition}

\begin{prop}
\label{mainprop}
If $j$ is a non-negative integer,
$$
D_n(P,YP^j,Y)=\sum_{J=0}^n \binom{j}{J}_P Y^J .
$$
\end{prop}

\begin{proof}
Let $J=k+i.$
The double sum $D_n(P,X,Y)$ can be rewritten as
$$
\begin{aligned}
D_n(P,X,Y) &=\sum_{J=0}^n \sum_{k=0}^J \frac{(-X)^k Y^{J-k}P^{-\binom{k}{2}}P^{-k(J-k)}}
{(P;P)_{J-k}(P;P)_k}\\
&=\sum_{J=0}^n \frac{Y^J}{(P;P)_J}\sum_{k=0}^J \binom{J}{k}_P P^{\binom{k}{2}} (-1)^k (XP^{1-J}/Y)^k.
\end{aligned}
$$

Now use the $q$-binomial theorem (page 78 of \cite{Br})
\begin{equation}
\label{qbt}
\sum_{k=0}^J \binom{J}{k}_P P^{\binom{k}{2}} (-1)^k Z^k=(Z;P)_J
\end{equation}
with $Z=XP^{1-J}/Y$ to evaluate the inner sum as $(XP^{1-J}/Y;P)_J.$ So
\begin{equation}
\label{singlesum}
D_n(P,X,Y)= \sum_{J=0}^n \frac{(XP^{1-J}/Y;P)_J}{(P;P)_J}Y^J
\end{equation}
If $X=YP^j$, then
$$
\frac{(XP^{1-J}/Y;P)_J}{(P;P)_J}=\binom{j}{J}_P
$$
\end{proof}

One may use Proposition~\ref{mainprop} with $j=-1$ (namely \eqref{singlesum} with $X=YP^{-1}$)
to obtain the next result.
\begin{prop}
\label{miniprop}
If $n$ is a non-negative integer,
$$
D_n(P,YP^{-1},Y)=\sum_{J=0}^n P^{-\binom{J+1}{2}} (-Y)^J .
$$
\end{prop}

Next, because the sum  in Proposition~\ref{mainprop} later occurs with
special choices of $Y$, we record three simple evaluations. The first identity
(\cite[p. 37]{A}) occurs for the unitary groups, while the last two (\cite[p. 49]{A})
will be helpful in treating the symplectic and orthogonal groups.

\begin{prop}
\label{andrews}
\begin{enumerate}
\item For any non-negative integer $j,$
$$
\sum_{J=0}^{2j} \binom{2j}{J}_q (-1)^J = (q;q^2)_j .
$$
\item For any non-negative integer $n,$
$$
\sum_{J=0}^{n} \binom{n}{J}_q q^{J/2} = (-q^{1/2};q^{1/2})_n.
$$
For any non-negative integer $n,$
\item
$$
\sum_{J=0}^n \binom{n}{J}_q q^{-J/2} = (-q^{1/2};q^{1/2})_n/q^{n/2}.
$$
\end{enumerate}
\end{prop}

Finally we record a generating function related to the double sum.
\begin{prop}
\label{genfunc} If $0<a<1,$ $|P|>1,$ and
$$
D_{nk}= \frac{(-X)^k}{ P^{\binom{k}{2}}(P;P)_k}
\sum_{i=0}^{n-k} \frac{P^{-ik}Y^i}
{(P;P)_{i}},
$$
then
$$
\begin{aligned}
F_k(a,X,Y;P) :=&(1-a)\sum_{n=k}^\infty D_{nk}a^n\\
= &(aYP^{-1};P^{-1})_\infty
\frac{(aX)^k P^{-k^2}}{(P^{-1};P^{-1})_k (aYP^{-1};P^{-1})_k}.
\end{aligned}
$$
\end{prop}

\begin{proof}  In the double sum definition of $F_k(a,X,Y;P),$ replace
$n$ by $n+k+i$ to obtain
$$
\begin{aligned}
F_k(a,X,Y;P)=&(1-a)\sum_{0\le n,i} a^{n+i}
\frac{(-aX)^k}{P^{\binom{k}{2}}(P;P)_k}
\frac{P^{-ik}Y^i}{(P;P)_i}\\
=&
\frac{(aX)^k P^{-k^2}}{(P^{-1};P^{-1})_k}
\sum_{i=0}^\infty \frac{P^{-\binom{i}{2}} (-YP^{-1-k})^i}
{(P^{-1};P^{-1})_i}a^i.
\end{aligned}
$$
Applying a limiting case of the $q$-binomial theorem
$$
\begin{aligned}
\sum_{i=0}^\infty \frac{P^{-\binom{i}{2}}}{(P^{-1};P^{-1})_i} (-P^{-1-k}Ya)^i=&
(aYP^{-1-k};P^{-1})_\infty\\
=& \frac{(aYP^{-1};P^{-1})_\infty}{(aYP^{-1};P^{-1})_k}
\end{aligned}
$$
completes the proof.
\end{proof}

\begin{remark} If $X=Y$, then $F_k(a,X,Y;P)$ may be summed by a limiting case of the
$q$-Gauss sum, \cite[(II.8)]{GaRa},
$$
\sum_{k=0}^\infty F_k(a,X,X;P)=1.
$$
This is the probability measure in Proposition~\ref{Al-Car}.
\end{remark}

\section{Main results} \label{main}

This section proves our main results about the distribution of the number of fixed vectors
of a random element of a finite classical group. We determine the limiting moments of these
distributions, and find precisely how large the rank of the group has to be in order
for the moment to stabilize to its limiting value.

Theorem \ref{GLmom} treats the general linear groups.

\begin{theorem}
\label{GLmom}
Let the random variable $Z_n$ be the number of fixed vectors of a random element of $GL(n,q)$
in its natural action. Then for all natural numbers $j$, and $n \geq j$, the expected value
of $Z_n^j$ is equal to the number of subspaces of a $j$-dimensional vector space over the
finite field $F_q$.
\end{theorem}

\begin{proof} Part 1 of Theorem \ref{RSGL} implies that
\begin{eqnarray*}
\begin{aligned}
E[Z_n^j] = & \sum_{k=0}^n P(Z_n=q^k) q^{jk}\\
= & \sum_{k=0}^n \frac{q^{jk}}{|GL(k,q)|} \sum_{i=0}^{n-k}
\frac{(-1)^i q^{{i \choose 2}}}{q^{ki} |GL(i,q)|} = D_n(q,q^j,1),
\end{aligned}
\end{eqnarray*}
where $D_n$ is given by Definition~\ref{doubledefn}.

Proposition~\ref{mainprop} implies
\[ D_n(q,q^j,1) = \sum_{J=0}^n \binom{j}{J}_q = \sum_{J=0}^{min(n,j)} \binom{j}{J}_q.\]
For $n \geq j$, this is the number of subspaces of a $j$-dimensional
vector space over the finite field $F_q$.
\end{proof}

\begin{remark} The number of subspaces of a $j$-dimensional vector space over the
finite field $F_q$ has been studied by Goldman and Rota \cite{GR}, who named these numbers
the Galois numbers $G_j$. They proved the recurrence

\[ G_{j+1} = 2G_j + (q^j-1) G_{j-1} .\]
\end{remark}

Another proof of Theorem \ref{GLmom} is in unpublished thesis work
\cite{F0} of the first named author, and goes as follows.

\begin{proof} Let $G=GL(n,q)$ and let $X$ be the product of $j$ copies of $V$, where $V$ is the
$n$-dimensional vector space on which $G$ acts. Let $G$ act on $X$ by acting separately
on each coordinate.

Consider the average number of fixed points of $G$ on $X$. On one hand, this is
the $j$th moment of the distribution of fixed vectors of $G$ on $V$. On the other
hand, by Burnside's lemma this is the number of orbits of $G$ on $X$. To each
orbit of $G$ on $X$, define an invariant $k$ of the orbit (called the number
of parts of the orbit) by taking any element $(v_1,\cdots,v_j)$ in the orbit
and letting $k$ be the dimension of the span of $v_1,\cdots,v_j$.

We claim that for $n \geq j$, the total number of orbits with invariant $k$
is equal to the $q$-binomial coefficient $\binom{j}{k}_q$, the number of $k$-dimensional
 subspaces of a $j$-dimensional vector space. This is proved bijectively.
Given an orbit, let $i_1,\cdots,i_k$ be the positions $i$ such that the
dimension of the span of $v_1,\cdots,v_i$ is one more than the dimension of
the span of $v_1,\cdots,v_{i-1}$. Let $(v_1,\cdots,v_j)$ be the unique element
of the orbit such that $v_{i_1},\cdots,v_{i_k}$ are the standard basis vectors
$e_1,\cdots,e_k$. Let $M$ be the $n \times j$ matrix whose columns are the
vectors $v_1,\cdots,v_j$. Let $M'$ be $M$ with the last $n-k$ rows chopped off,
so that $M'$ is a $k \times j$ matrix. Note that $M'$ is in reduced row-echelon
form, and hence by basic linear algebra corresponds to a unique $k$-dimensional
subspace of a $j$-dimensional space.
\end{proof}

Next we treat the unitary groups $U(n,q)$. Since $U(n,q)$ is viewed as a subgroup
of $GL(n,q^2)$, an element of $U(n,q)$ with a $k$-dimensional fixed space has $q^{2k}$ many
fixed vectors.

\begin{theorem} \label{Umom}
Let the random variable $Z_n$ be the number of fixed vectors of a random element of $U(n,q)$
in its natural action. Then for all natural numbers $j$, and $n \geq 2j$, the expected value
of $Z_n^j$ is equal to \[ \prod_{i=1}^j (q^{2i-1}+1) .\]
\end{theorem}

\begin{proof} Let $Q=-q$. Theorem~\ref{RSU}(1) yields
\begin{eqnarray*}
E[Z_n^j] & = & \sum_{k=0}^n \frac{q^{2jk}}{|U(k,q)|} \sum_{i=0}^{n-k}
\frac{(-1)^i (-q)^{{i \choose 2}}}{(-q)^{k
i} |U(i,q)|}
 =D_n(Q,-Q^{2j},-1).
\end{eqnarray*}
By Proposition~\ref{mainprop} and Proposition~\ref{andrews}(1) we have
$$
\begin{aligned}
D_n(Q,-Q^{2j},-1) &=\sum_{J=0}^n  \binom{2j}{J}_Q(-1)^J \\
&= (Q;Q^2)_j = \prod_{i=1}^j (q^{2i-1}+1) {\text{ if $n\ge 2j$.}}
\end{aligned}
$$
\end{proof}

Next we treat the symplectic groups.

\begin{theorem}
\label{Spmom}
Let the random variable $Z_n$ be the number of fixed vectors of a random element of $Sp(2n,q)$
in its natural action. Then for all natural numbers $j$, and $n \geq j$, the expected value
of $Z_n^j$ is equal to \[ \prod_{i=1}^j (q^{i-1}+1) .\]
\end{theorem}

\begin{proof} We treat the cases $j=0$ and $j>0$ separately.

Suppose first that $j=0$. From parts 1 and 2 of Theorem \ref{RSSp}, there
are two double sums to consider, and let $R=q^2$. The first
double sum is
\begin{eqnarray*}
 \sum_{k=0}^n \frac{1}{|Sp(2k,q)|} \sum_{i=0}^{n-k} \frac{(-1)^i q^{i(i+1)}}{|Sp(2i,q)| q^{2ik}}
=  D_n(R,q^{-1},q)=\sum_{J=0}^n (-1)^J q^{-J^2}
\end{eqnarray*}
where we have used Proposition~\ref{miniprop}.

The second double sum is
\begin{eqnarray*}
 \sum_{k=0}^{n-1} \frac{1}{q^{2k+1} |Sp(2k,q)|} \sum_{i=0}^{n-k-1} \frac{(-1)^i q^{i(i+1)}}{q^{2i(k+1)} |Sp(2i,q)|}
 =\frac{1}{q} D_{n-1}(R,q^{-3},q^{-1}).
\end{eqnarray*}
which by Proposition~\ref{miniprop} is
\begin{eqnarray*}
\sum_{J=0}^{n-1} (-1)^J q^{-(J+1)^2}.
\end{eqnarray*}
The $j=0$ case of the theorem now follows since
\[ \sum_{J=0}^n (-1)^J q^{-J^2} + \sum_{J=0}^{n-1} (-1)^J q^{-(J+1)^2} = 1 .\]

Next we consider the case that $j \geq 1$.
Again by parts 1 and 2 of Theorem \ref{RSSp}, there are two double sums.
The first double sum is
\begin{eqnarray*}
& & \sum_{k=0}^n \frac{q^{2kj}}{|Sp(2k,q)|} \sum_{i=0}^{n-k} \frac{(-1)^i q^{i(i+1)}}{|Sp(2i,q)|q^{2ik}} =
D_n(R,q^{2j-1},q) =\sum_{J=0}^n \binom{j-1}{J}_R R^{\frac{J}{2}}.
\end{eqnarray*}
By Proposition~\ref{andrews}(2), this is equal to $(-q;q)_{j-1}$ if $n \geq j-1$.

Again using Proposition~\ref{mainprop} and Proposition~\ref{andrews}(3),
the second double sum is
$$
\begin{aligned}
 &\sum_{k=0}^{n-1} \frac{q^{(2k+1)j}}{|Sp(2k,q)| q^{2k+1}}  \sum_{i=0}^{n-k-1}
 \frac{(-1)^i q^{i(i+1)}}{q^{2i(k+1)}} \frac{1}{|Sp(2i,q)|} \\
 &= q^{j-1} D_{n-1}(R,q^{2j-3},q^{-1})
 =q^{j-1} \sum_{J=0}^{n-1} \binom{j-1}{J}_R R^{-J/2}\\
 & = (-q;q)_{j-1} { \text{ if $n\ge j$.}}
\end{aligned}
$$

The theorem now follows since if $n\ge j,$
\[ E[Z_n^j] = 2 (-q;q)_{j-1} = 2(1+q) \cdots (1+q^{j-1}) = \prod_{i=1}^j (q^{i-1}+1).\]
\end{proof}

\begin{remark} The $j=0$ case of Theorem~\ref{Spmom} follows immediately since
any probability distribution sums to $1$. However we feel that the proof
of this case in the proof of Theorem~\ref{Spmom} is of interest.
\end{remark}

Next we treat the odd characteristic, odd dimensional orthogonal groups $O(2n+1,q)$. As explained just
before the proof of Theorem \ref{RSOodd}, there is no need to treat odd dimensional orthogonal groups
in even characteristic.

\begin{theorem}
\label{OddOmom}
Suppose that the characteristic is odd.
Let the random variable $Z_n$ be the number of fixed vectors of a random element of $O(2n+1,q)$
in its natural action. Then for all natural numbers $j$, and $n \geq j$, the expected value
of $Z_n^j$ is equal to \[ \prod_{i=1}^j (q^i+1) .\]
\end{theorem}

\begin{proof} The proof is nearly the same as that of Theorem~\ref{Spmom}, so we abbreviate the details.
From parts 1 and 2 of Theorem \ref{RSOodd}, there are two double sums, and we let $R=q^2$.
The first double sum is
\begin{eqnarray*}
 \frac{1}{2} D_n(R,q^{2j+1},q) =\frac{1}{2} \sum_{J=0}^n \binom{j}{J}_R R^{J/2}.
\end{eqnarray*}
If $n \geq j$, then this is equal to $\frac{1}{2} (-q;q)_j$ by Proposition~\ref{andrews}(2).

The second double sum is
\begin{eqnarray*}
 \frac{1}{2} q^j D_n(R,q^{2j-1},1/q)= \frac{1}{2} q^j \sum_{J=0}^n \binom{j}{J}_R R^{-J/2}.
\end{eqnarray*}
If $n \geq j$, then this is equal to $\frac{1}{2}(-q;q)_j$ by Proposition~\ref{andrews}(3).

So the two double sums are equal, and adding their values gives the required
$(-q;q)_j = \prod_{i=1}^j (q^i+1).$
\end{proof}

Next we consider even dimensional orthogonal groups. As mentioned in Section \ref{RSreview}, the
formulas are the same in even and odd characteristic. We first treat $O^+(2n,q)$, and then
treat $O^-(2n,q)$. Note that the moments stabilize more quickly for $O^+(2n,q)$ than for
$O^-(2n,q)$.

\begin{theorem}
\label{evenO+}
Let the random variable $Z_n$ be the number of fixed vectors of a
random element of $O^+(2n,q)$ in its natural action. Then for all natural numbers $j$, and
$n \geq j$, the expected value of $Z_n^j$ is equal to \[ \prod_{i=1}^j (q^i+1) .\]
\end{theorem}

\begin{proof} From parts 1 and 2 of Theorem \ref{RSOeven}, there are two double sums and a
single sum, and we let $R=q^2$.

The first double sum is:
\begin{eqnarray*}
\frac{1}{2} D_n(R,q^{2j+1},q),
\end{eqnarray*} which is equal to $\frac{1}{2} (-q;q)_j$ if $n \geq j$.

Next, for the single sum one computes that
\begin{eqnarray*}
& & \frac{1}{2} \sum_{k=0}^n \frac{(-1)^{n-k} q^{2kj}}{q^{2k(n-k)} |GL(k,q^2)| (q^{2(n-k)}-1) \cdots (q^2-1)} \\
& = & \frac{1}{2} \frac{1}{(R;R)_n} \sum_{k=0}^n \binom{n}{k}_R (-1)^k R^{{k \choose 2}} (R^{j+1-n})^k \\
& = & \frac{1}{2} \frac{1}{(R;R)_n} (R^{j+1-n};R)_n,
\end{eqnarray*} where the last step used the $q$-binomial theorem (\ref{qbt}). Note that
\[ \frac{1}{2} \frac{1}{(R;R)_n} (R^{j+1-n};R)_n \] is equal to $\frac{1}{2}$ if $n=j$, and to $0$ if $n>j$.
Summarizing, the even dimensional fixed spaces contribute
\[ \left\{ \begin{array}{ll}
\frac{1}{2} \left[ (-q;q)_j + 1 \right] & \mbox{if $n=j$}\\
\frac{1}{2} \left[ (-q;q)_j \right] & \mbox{if $n>j$}\end{array}
                        \right.  \]

Next consider the contribution from the odd dimensional fixed spaces in Theorem~\ref{RSOeven}.
It is
\begin{eqnarray*}
\frac{q^j}{2} D_{n-1}(R,q^{2j-1},q^{-1})=\frac{q^j}{2} \sum_{J=0}^{n-1} \binom{j}{J}_R R^{-J/2}.
\end{eqnarray*}
By Proposition~\ref{andrews}(3),
this is equal to $\frac{1}{2}(-q;q)_j$ if $n>j$, and to $\frac{1}{2}[(-q;q)_j-1] $ if $n=j$.
Summarizing, this contribution is
\[ \left\{ \begin{array}{ll}
\frac{1}{2} \left[ (-q;q)_j - 1 \right] & \mbox{if $n=j$}\\
\frac{1}{2} \left[ (-q;q)_j \right] & \mbox{if $n>j$}\end{array}
                        \right.  \]

Adding the values of the two contributions completes the proof.
\end{proof}

\begin{theorem} \label{evenO-} Let the random variable $Z_n$ be the number of fixed vectors of a
random element of $O^-(2n,q)$ in its natural action. Then for all natural numbers $j$, and
$n \geq j+1$, the expected value of $Z_n^j$ is equal to \[ \prod_{i=1}^j (q^i+1) .\]
\end{theorem}

\begin{proof} The proof is very similar to that of Theorem~\ref{evenO+}.
\end{proof}

\section{Connections to classical orthogonal polynomials}
\label{orthog}

The Charlier polynomials $C_n(x;a)$ are orthogonal with respect to the Poisson distribution
$$
w_{Char}(k,a)= e^{-a} \frac{a^k}{k!}, \quad k=0,1, \cdots .
$$
In this section we shall explain
how the $n\to\infty$ limiting case of the moments described by
Theorems~\ref{GLmom}-\ref{evenO-} are related to $q$-analogues of Charlier polynomials.
We shall use using Al-Salam-Carlitz polynomials and the
$q$-Charlier polynomials.

First we use the Al-Salam-Carlitz polynomials, see \cite[p. 195-198]{Ch}.

\begin{prop}
\label{Al-Car}
Let $0<p<1$ and $a>0.$ The Al-Salam-Carlitz polynomials $V_n^{(a)}(x;p)$ are
orthogonal with respect to the discrete probability measure which is
supported on the sequence $p^{-k}$ with masses of
$$
w_{AC}(p^{-k};a;p)\ =\ (ap;p)_\infty \frac{p^{k^2}a^k}{(p;p)_k (ap;p)_k}, \quad k=0,1,\cdots, .
$$
The moments $\mu_j$ of this measure are given by \cite[(10.10)]{Ch}
$$
\mu_j=\sum_{k=0}^j \binom{j}{k}_{1/p} a^k.
$$
\end{prop}

We see that the limiting measure in Theorem~\ref{RSGL} corresponds to
$p=1/q$ and $a=1$ in Proposition~\ref{Al-Car}.

\begin{cor} The moments for the Al-Salam-Carlitz polynomials for $p=1/q$ and
$a=1$ are
the Galois numbers
$$
\mu_j=G_j.
$$
\end{cor}

\begin{remark} Note the following limiting case
of the Al-Salam-Carlitz weight to the Poisson distribution
$$
\lim_{p\to 1} w_{AC}(p^{-k},(1-p)a;p)=e^{-a}\frac{a^k}{k!}.
$$
\end{remark}

To compute the moments in the other limiting cases
we use a $q$-Charlier polynomial \cite[Exer. 7.13, p. 202]{GaRa},
whose measure is also purely discrete.
\begin{prop}
\label{q-Char}
Let $0<p<1$ and $a>0.$ The $p$-Charlier polynomials $C_n(x;a,p)$ are orthogonal with respect to
the discrete probability measure which is supported on the sequence $p^{-k}$ with masses of
$$
w_{p-Char}(p^{-k};a,p)\ =\
\frac{1}{(-a;p)_\infty} \frac{p^{\binom{k}{2}}}{(p;p)_k}a^k, \quad k=0,1,\cdots, .
$$
The moments $\mu_j$ of this measure are given by \cite[\S7]{DSW}
$$
\mu_j=(-a/p;1/p)_j,
$$
\end{prop}

\begin{remark} Note the following limiting case
of the $q$-Charlier weight to the Poisson distribution
$$
\lim_{p\to 1} w_{p-Char}(p^{-k},(1-p)a;p)=e^{-a}\frac{a^k}{k!}.
$$
\end{remark}

The remaining limiting cases of \S\ref{main} are
\begin{enumerate}
\item Theorem~\ref{Umom}, proven from Theorem~\ref{RSU}(2) and Proposition~\ref{q-Char} with $p=1/q^2$, $a=1/q$,
\item Theorem~\ref{Spmom}, proven from Theorem~\ref{RSSp}(3) and Proposition~\ref{q-Char} with $p=1/q$, $a=1/q$,
\item Theorem~\ref{OddOmom}, proven from Theorem~\ref{RSOodd}(3) and Proposition~\ref{q-Char} with $p=1/q$, $a=1$,
\item Theorems~\ref{evenO+} and \ref{evenO-}, proven from
Theorem~\ref{RSOeven}(3) and Proposition~\ref{q-Char} with $p=1/q$, $a=1.$
\end{enumerate}

\begin{remark}
The moments for the Charlier polynomial measure are
$$
\mu_j^{Char}=\sum_{k=1}^j S(j,k)a^k,
$$
where $S(j,k)$ are Stirling numbers
of the second kind. They count set partitions of $\{1,2,\cdots,j\}$
by the number of blocks $k$ and refine the Bell numbers.

If the Al-Salam-Carlitz polynomials are rescaled so that the measure is located at
$(q^{-k}-1)/(1-q)$, which has a limiting value of $k$ as $q\to 1,$
and $a$ is replaced by $(1-q)a,$ the moments are given by
$q$-Stirling numbers (see \cite[(3.2)]{DSW})
$$
\mu_j^{AC}=q^{-j}\sum_{k=1}^j  q^{k}S_{1/q}(j,k) a^k.
$$

If the $q$-Charlier polynomials are rescaled so that the measure is located at
$(q^{-k}-1)/(1-q)$, which has a limiting value of $k$ as $q\to 1,$
and $a$ is replaced by $(1-q)a,$
the moments are given by $q$-Stirling numbers \cite[(7.4)]{DSW}
$$
\mu_j^{q-Char}=q^{-j} \sum_{k=1}^j  q^{-\binom{k}{2}}S_{1/q}(j,k) a^k.
$$
\end{remark}

All values of $n$ may be considered simultaneously by choosing a random $n$.
We use the fact that the Al-Salam-Carlitz weight satisfies
$$
F_k(a,X,X;P)=w_{AC}(p^{-k};aX;p), \quad P=1/p.
$$

\begin{theorem} Suppose that a non-negative integer $n$ is chosen with probability $(1-a)a^n.$
Then the probability that a random element of a random $GL(n,q)$
has $q^k$ fixed vectors is given by the probability measure in Proposition~\ref{Al-Car} with $p=1/q.$
\end{theorem}

\begin{proof} We need
$$
\sum_{n=0}^\infty P(Z_n=q^k)a^n(1-a)=F_k(a,1,1;q)=
w_{AC}(p^{-k};a;p), \quad p=1/q.
$$
This result also appears in \cite{F1}.
\end{proof}

For the unitary groups, note that the proof of Theorem~\ref{Umom} used the values
$D_n(Q,-Q^{2j},-1)$, $Q=-q.$ So we need
$$
\begin{aligned}
\sum_{n=0}^\infty P(Z_n=q^{2k})a^n(1-a)&=F_k(a,-1,-1;Q)\\
&=
w_{AC}(p^{-k};-a;p), \quad p=1/Q=-1/q.
\end{aligned}
$$

\begin{theorem} Suppose that a non-negative integer $n$ is chosen with probability $(1-a)a^n.$
Then the probability that a random element of a random $U(n,q)$
has $q^{2k}$ fixed vectors is given by the probability measure in
Proposition~\ref{Al-Car} with $p=-1/q$ and
$a$ replaced by $-a.$
\end{theorem}

For the symplectic groups, there are two cases in
Theorem~\ref{RSSp}, depending upon the fixed point space being even or odd
dimensional. Curiously, these two cases may be combined into a single example of
the Al-Salam-Carlitz weight. This is analogous to the association scheme
of symmetric matrices, where adjacent ranks are combined depending on the parity,
see \cite{Eg}.

\begin{theorem} Suppose that a non-negative integer $n$ is chosen with probability $(1-a)a^n.$
Then the probability that a random element of a random $Sp(2n,q)$
has $q^{2k}$ or $q^{2k+1}$ fixed vectors is given by the probability measure in
Proposition~\ref{Al-Car} with $p=1/q^2$ and $a$ replaced by $a/q.$
\end{theorem}

\begin{proof} The proof of Theorem~\ref{Spmom} used two terms for $j=0$,
$$
D_n(R, q^{-1},q)+ \frac{1}{q} D_{n-1}(R,q^{-3},q^{-1}).
$$
So we need for $P=R=q^2$,
$$
\begin{aligned}
\sum_{n=0}^\infty &\left( P(Z_n=q^{2k})+  P(Z_n=q^{2k+1})\right)
a^n(1-a)\\
&=  F_k(a, q^{-1},q;P)+\frac{a}{q}F_k(a,q^{-3},q^{-1};P)\\
&=
(aq^{-1}P^{-1};P^{-1})_\infty
\frac{(aq^{-1})^k P^{-k^2}}{(P^{-1};P^{-1})_k (aq^{-1}P^{-1};P^{-1})_k}\\
&= w_{AC}(P^k;aq^{-1};P^{-1}).
\end{aligned}
$$
\end{proof}

\section{Acknowledgements} Fulman was partially supported by NSA grant H98230-13-1-0219.
Stanton was partially supported by NSF grant DMS-1148634.
The authors thank Persi Diaconis for helpful discussions.

\end{document}